\documentclass[12pt]{amsart}
\usepackage[colorlinks=true,citecolor=black,linkcolor=black,urlcolor=blue]{hyperref}
\usepackage{amsmath}
\usepackage[english,  activeacute]{babel}
\usepackage[utf8]{inputenc}
\usepackage{amssymb}
\usepackage{amsthm}
\usepackage{graphics,graphicx}
\usepackage{enumerate}
\usepackage{ytableau}
\usepackage{array}
\usepackage{bm}
\usepackage{cite}
\usepackage{a4wide}
\usepackage{color, url}
\usepackage{float}
\usepackage{comment}
\usepackage{xcolor}
\usepackage{color, url}
\usepackage{float}
\usepackage{pgf, tikz}

\theoremstyle{plain}
\newtheorem{theorem}{Theorem}[section]
\newtheorem{proposition}[theorem]{Proposition}
\newtheorem{corollary}[theorem]{Corollary}
\newtheorem{lemma}[theorem]{Lemma}

\theoremstyle{definition}

\usepackage[usestackEOL]{stackengine}[2013-10-15]
\stackMath

\begin{document}


\title{Log-concavity and strong $q$-log-convexity for some generalized triangular arrays}
\author[B. Rezig]{Boualam Rezig}
\address{High normal school Constantine, LMAM laboratory, Jijel, Algeria}
\email{boualem.rezig@gmail.com }
\author[M. Ahmia]{Moussa Ahmia}
\address{University of Mohamed Seddik Benyahia, LMAM laboratory, Jijel, Algeria}
\email{moussa.ahmia@univ-jijel.dz;
ahmiamoussa@gmail.com}

\maketitle

\begin{abstract}
In this paper, we provide criteria for the log-concavity of rows and the strong $q$-log-convexity of the generating functions of rows in more generalized triangles. Additionally, we prove that the bi$^s$nomial transformation not only preserves the strong $q$-log-convexity property but also preserves the strong $q$-log-concavity property.
\end{abstract}

\section{Introduction}
\label{Sec:1}
Recall that a sequence of nonnegative real numbers $(x_{n})_{n\geq 0}$ is called log-concave if $x_{n}^{2}\geq x_{n-1}x_{n+1}$, and log-convex if $x_{n}^{2}\leq x_{n-1}x_{n+1}$, for all $n\geq 1$. These conditions are equivalent to having $x_{n}x_{m}\geq x_{n-1}x_{m+1}$ in the log-concave case, and $x_{n}x_{m}\leq x_{n-1}x_{m+1}$ in the log-convex case, for all $m\geq n\geq 1$. Log-concave and log-convex sequences frequently arise in combinatorics, algebra, geometry, analysis, probability, and statistics, and have been extensively investigated; see Stanley \cite{St} and Brenti \cite{br} for log-concavity, and Liu and Wang \cite{LI} for log-convexity.

\medskip

For two polynomials with real coefficients $f(q)$ and $g(q)$, we denote $f(q) \geq _{q}g(q)$ if the difference $f(q)-g(q)$ has only nonnegative coefficients. A polynomial sequence $\left(f_{n}(q)\right)_{n\geq 0}$ is called $q$-log-concave (respectively $q$-log-convex, as introduced by Liu and Wang \cite{LI} and first suggested by Stanley \cite{St}) if
\begin{equation*}
f_{n}(q)^{2}\geq _{q}f_{n-1}(q)f_{n+1}(q) \text{ \ \ (respectively, $f_{n}(q)^{2} \leq_{q} f_{n-1}(q)f_{n+1}(q)$)},
\end{equation*}%
for $n\geq 1$. It is called strongly $q$-log-concave (respectively, strongly $q$-log-convex) if%
\begin{equation*}
f_{n}(q)f_{m}(q) \geq _{q} f_{n-1}(q)f_{m+1}(q)\text{ \ \ (respectively, $f_{n}(q)f_{m}(q)\leq _{q} f_{n-1}(q)f_{m+1}(q)$)},
\end{equation*}%
for $m\geq n\geq 1$ (see Chen et al \cite{Chn}). Clearly, the strong $q$-log-concavity (respectively, strong $q$-log-convexity) of polynomial sequences implies the $q$-log-concavity (respectively, $q$-log-convexity). However, the converse does not hold (see Chen et {\it al.} \cite{Chn}). It is straightforward to observe that if the sequence $\{f_{n}(q)\}_{n \geq 0}$ is $q$-log-concave (or $q$-log-convex), then for each fixed nonnegative number $q$, the sequence $\{f_{n}(q)\}_{n \geq 0}$ is log-concave (or log-convex, respectively). The $q$-log-concavity and $q$-log-convexity of polynomials have been extensively studied; see, for instance, \cite{But,Bu,Ch,Cn,Chn,Do,Dd,Ler,Lii,LI,Sag,Sg,WSu,XZ,XZZ}.

\medskip

Let $A=[a_{n,k}]_{n,k\geq 0}$ be an infinite matrix. It is called totally positive of order $r$ (\textbf{TP}$_{r}$, for short) if all of its minors
of  orders $\leq r$ are nonnegative. It is called TP if all of its minors, of any order, are nonnegative. For example, the Pascal triangle is a TP matrix \cite%
[p. 137]{kar}. Totally positive matrices play an important role in the theory of total positivity \cite{fa,kar,pin}. Let $\{a_{n}\}_{n\geq 0}$ be an infinite sequence of nonnegative numbers. This sequence is called a P\'{o}lya frequency sequence of order $r$ (a \textbf{PF}$_{r}$ sequence, for short) if its Toeplitz matrix
\begin{equation*}
\left( a_{i-j}\right) _{i,j\geq 0}=%
\begin{bmatrix}
a_{0}\text{ \ } &  &  &  &  &  \\
a_{1}\text{ \ } & a_{0}\text{ \ } &  &  &  &  \\
a_{2}\text{ \ } & a_{1}\text{ \ } & a_{0}\text{ \ } &  &  &  \\
a_{3}\text{ \ } & a_{2}\text{ \ } & a_{1}\text{ \ } & a_{0}\text{ \ } &  &
\\
a_{4}\text{ \ } & a_{3}\text{ \ } & a_{2}\text{ \ } & a_{1}\text{ \ } & a_{0}%
\text{ \ } &  \\
\vdots \text{ \ } &  &  &  &  & \ddots \text{ \ }%
\end{bmatrix},
\end{equation*}%
is TP$_{r}$. A sequence $\{a_{n}\}$ is called PF if its Toeplitz matrix is TP. Specifically, the sequence $\{a_n\}$ is log-concave if and only if it is PF$_{2}$, meaning that its Toeplitz matrix $[a_{i-j}]_{i,j\geq 0}$ is TP$_{2}$. Similarly, the sequence is log-convex if and only if its Hankel matrix
\begin{equation*}
\left( a_{i+j}\right) _{i,j\geq 0}=%
\begin{bmatrix}
a_{0}\text{ \ } & a_{1}\text{ \ } & a_{2}\text{ \ } & a_{3}\text{ \ } & a_{4}%
\text{ \ } & \cdots \text{ \ } \\
a_{1}\text{ \ } & a_{2}\text{ \ } & a_{3}\text{ } & a_{4}\text{ \ } & a_{5}%
\text{ \ } & \cdots \text{ \ } \\
a_{2}\text{ \ } & a_{3}\text{ \ } & a_{4}\text{ \ } & a_{5}\text{ \ } & a_{6}%
\text{ \ } & \cdots \text{ \ } \\
\vdots \text{ \ } &  &  &  &  & \ddots \text{ }\ \ \text{\ }%
\end{bmatrix},
\end{equation*}%
is TP$_{2}$ \cite{br}.

\medskip

Let $n$ and $s$ be two positive integers. The bi$^{s}$nomial coefficients $\binom{n}{k}_{s}$ are defined as the $k^{\text{th}}$ coefficients in the expansion of
\begin{equation*}
(1+x+\cdots+x^{s})^{n}=\sum_{k=0}^{sn}\binom{n}{k}_{s}x^{k},
\end{equation*}%
where $\binom{n}{k}_{s}=0$ unless $sn\geq k\geq 0$.

\medskip

They satisfy the following relations \cite{AMH,bsb,bro}:
\begin{itemize}
\item the symmetry relation

\begin{equation}
\binom{n}{k}_{s}=\binom{n}{sn-k}_{s},  \label{eq3}
\end{equation}

\item the longitudinal recurrence relation

\begin{equation}
\binom{n}{k}_{s}=\sum_{j=0}^{s}\binom{n-1}{k-j}_{s}.  \label{eq}
\end{equation}
\end{itemize}
These coefficients, as for usual binomial coefficients, are built trough the Pascal triangle, known as \textquotedblleft $s$-Pascal triangle
\textquotedblright or generalized Pascal triangle. One can find the first values of this triangle in OEIS \cite{slo} as A027907 for $s=2$.

\medskip

It is a good way to determine log-concavity or log-convexity using various operators. For instance, Davenport and Pólya demonstrated that log-convexity is preserved under binomial convolution \cite{Dvp}. Similarly, Wang and Yeh proved that log-concavity is also preserved under binomial convolution \cite{YW}. In addition, Ahmia and Belbachir established both properties under the bi$^s$nomial convolution \cite{abel, AH}. Liu and Wang further explored the log-convexity of several well-known combinatorial sequences including Catalan numbers, Motzkin numbers, Fine numbers, central Delannoy numbers, and Schröder numbers by examining their recurrences and linear transformations \cite{LI}.

\medskip

There are many combinatorial sequences arising from triangular arrays, such as the Pascal triangle, Stirling triangle, Aigner's Catalan triangle, Shapiro's Catalan triangle, Motzkin triangle, and Bell triangle.

\medskip

In fact, the triangles discussed above can be derived using a unified approach as follows. Let $\left[C_{n,k}\right]_{n,k\geq 0}$ be an infinite lower triangular array that satisfies the recurrence relation:
\begin{equation}\label{eq1}
C_{n,k}=C_{n-1,k-1}+f_{k}C_{n-1,k}+g_{k}C_{n-1,k+1}
\end{equation}
for $n\geq 1$ and $0\leq k \leq n$, where $C_{0,0}=1$ and $C_{0,k}=A_{0,-1}=0$ for $k>0$. This array$\left[C_{n,k}\right]_{n,k\geq 0}$ is referred to as the recursive matrix, and the values $C_{n,0}$ are known as the Catalan-like numbers \cite{XZ,XZ1}. Chen et {\it al.} \cite{Chn1} considered the total positivity of recursive matrices. In \cite{XZ}, Zhu provided a criterion for the log-convexity of Catalan-like numbers, specifically the first column of the recursive matrices $\left[C_{n,k}\right]_{0\leq k\leq n}$.

\medskip

The log-concavity of row sequences and the strong $q$-log-convexity of the generating functions of the rows in the aforementioned triangles have been established by Zhu \cite{XZ1}. This author demonstrated that the binomial transformation preserves both strong $q$-log-convexity and strong $q$-log-concavity properties. Additionally, he proved that the strong $q$-log-convexity property is maintained under the Stirling transformation and the Whitney transformation of the second kind.

\medskip

Building on previous works, Section \ref{Sec:2} presents a sufficient condition for the log-concavity of rows in more generalized triangles. As applications of this result, we demonstrate that each row sequence in the $2$-Pascal triangle, as well as in the triangles \href{https://oeis.org/A291082}{A291082} and \href{https://oeis.org/A291080}{A291080} in OEIS \cite{slo}, is log-concave. In Section \ref{Sec:3}, we explore some linear transformations that preserve the strong $q$-log-convexity property. Finally, in Section \ref{Sec:4}, we prove that the bi$^{s}$nomial transformation not only preserves the strong $q$-log-convexity property but also the strong $q$-log-concavity property.
\section{Log-concavity of the rows}
\label{Sec:2}
In this section, we will present a sufficient condition for the log-concavity of the rows in recursive matrices
\begin{theorem}\label{thhm}
Let an infinite lower triangular array $\left[A_{n,k}\right]_{n,k\geq 0}$ satisfy the recurrence
\begin{align}\label{eqq1}
&A_{n,k}=\gamma_{k} A_{n-1,k-2}+e_{k}A_{n-1,k-1}+f_{k}A_{n-1,k}+g_{k}A_{n-1,k+1}+h_{k}A_{n-1,k+2}
\end{align}
for $n\geq 1$ and $0\leq k \leq 2n$, with the initial conditions $A_{0,0}=1$ and $A_{0,k}=A_{0,-1}=0$ for $k>0$. Assume that the nonnegative sequences $\left(\gamma_{k}\right)_{k\geq 2}$ $\left(e_{k}\right)_{k\geq 1}$, $\left(f_{k}\right)_{k\geq 0}$, $\left(g_{k}\right)_{k\geq 0}$ and $\left(h_{k}\right)_{k\geq 0}$ are all log-concave. If
\begin{description}
  \item[(1)] $2\gamma_{k}e_{k}\geq \gamma_{k-1}e_{k+1}+\gamma_{k+1} e_{k-1}$,
  \item[(2)] $2\gamma_{k}f_{k}\geq \gamma_{k-1}f_{k+1}+\gamma_{k+1} f_{k-1}$,
  \item[(3)] $2\gamma_{k}g_{k}\geq \gamma_{k-1}g_{k+1}+\gamma_{k+1} g_{k-1}$,
  \item[(4)] $2\gamma_{k}h_{k}\geq \gamma_{k-1}h_{k+1}+\gamma_{k+1} h_{k-1}$,
  \item[(5)] $2e_{k}f_{k}\geq e_{k+1}f_{k-1}+e_{k-1}f_{k+1}$ and $e_{k+1}e_{k-1}\geq \gamma_{k+1} f_{k-1}$,
  \item[(6)] $2e_{k}g_{k}\geq e_{k+1}g_{k-1}+e_{k-1}g_{k+1}$ and $f_{k+1}e_{k-1}\geq \gamma_{k+1}g_{k-1}$,
  \item[(7)] $2e_{k}h_{k}\geq e_{k+1}h_{k-1}+e_{k-1}h_{k+1}$ and $g_{k+1}e_{k-1}\geq \gamma_{k+1}h_{k-1}$,
  \item[(8)] $2f_{k}g_{k}\geq f_{k+1}g_{k-1}+f_{k-1}g_{k+1}$ and $f_{k+1}f_{k-1}\geq e_{k+1}g_{k-1}$,
  \item[(9)] $2f_{k}h_{k}\geq f_{k+1}h_{k-1}+f_{k-1}h_{k+1}$ and $g_{k+1}f_{k-1}\geq e_{k+1}h_{k-1}$,
  \item[(10)] $2g_{k}h_{k}\geq g_{k+1}h_{k-1}+g_{k-1}h_{k+1}$ and $g_{k+1}g_{k-1}\geq f_{k+1}h_{k-1}$
\end{description}
for all $k\geq 2$, then for any fixed $n$, the sequence $\left[A_{n,k}\right]_{0\leq k\leq 2n}$ is log-concave
in $k$.
\end{theorem}
\begin{proof}
To demonstrate that $\left[A_{n,k}\right]_{0 \leq k \leq 2n}$ is log-concave in $k$, it is sufficient to prove that
\begin{equation*}
A_{n,k}^{2}-A_{n,k+1}A_{n,k-1}\geq 0
\end{equation*}
for any $k \geq 0$, which will be prove by induction on $n$. It is clear for $n = 0$. Thus, we assume the statement holds for $1\leq n \leq m$. Now, for $n=m+1$ and $0\leq k \leq m+1$, we have
\begin{align}   \label{eqq2}
&A_{m+1,k}^{2}-A_{m+1,k+1}A_{m+1,k-1}   \notag \\
=&\left[\gamma_{k}A_{m,k-2}+e_{k}A_{m,k-1}+f_{k}A_{m,k}+g_{k}A_{m,k+1}+h_{k}A_{m,k+2}\right]^{2} \notag \\
   &-\left[\gamma_{k+1}A_{m,k-1}+e_{k+1}A_{m,k}+f_{k+1}A_{m,k+1}+g_{k+1}A_{m,k+2}+h_{k+1}A_{m,k+3}\right] \notag \\
   &\times \left[\gamma_{k-1}A_{m,k-3}+e_{k-1}A_{m,k-2}+f_{k-1}A_{m,k}+g_{k-1}A_{m,k}+h_{k-1}A_{m,k+1}\right]  \notag \\ =&\left[{\gamma_{k}^{2}A_{m,k-2}^{2}-\gamma_{k+1}\gamma_{k-1}A_{m,k-1}A_{m,k-3}}\right]+\left[{e_{k}^{2}A_{m,k-1}^{2}-e_{k+1}e_{k-1}A_{m,k}A_{m,k-2}}\right]  \notag \\
   &+\left[{f_{k}^{2}A_{m,k}^{2}-f_{k+1}f_{k-1}A_{m,k+1}A_{m,k-1}}\right]+\left[{g_{k}^{2}A_{m,k+1}^{2}-g_{k+1}g_{k-1}A_{m,k+2}A_{m,k}}\right]\notag \\ &+\left[{h_{k}^{2}A_{m,k+2}^{2}-h_{k+1}h_{k-1}A_{m,k+3}A_{m,k+1}}\right]\notag \\
   &+\left[{2\gamma_{k}e_{k}A_{m,k-1}A_{m,k-2}-\gamma_{k-1}e_{k+1}A_{m,k}A_{m,k-3}-\gamma_{k+1}e_{k-1}A_{m,k-1}A_{m,k-2}}\right]\notag\\
   &+\left[{2\gamma_{k}f_{k}A_{m,k}A_{m,k-2}-\gamma_{k-1}f_{k+1}A_{m,k+1}A_{m,k-3}-\gamma_{k+1}f_{k-1}A^{2}_{m,k-1}}\right]\notag\\   
   &+\left[{2\gamma_{k}g_{k}A_{m,k+1}A_{m,k-2}-\gamma_{k-1}g_{k+1}A_{m,k+2}A_{m,k-3}-\gamma_{k+1}g_{k-1}A_{m,k}A_{m,k-1}}\right]\notag\\
   &+\left[{2\gamma_{k}h_{k}A_{m,k+2}A_{m,k-2}-\gamma_{k-1}h_{k+1}A_{m,k+3}A_{m,k-3}-\gamma_{k+1}h_{k-1}A_{m,k+1}A_{m,k-1}}\right]\notag \\
   &+\left[{2e_{k}f_{k}A_{m,k}A_{m,k-1}-f_{k+1}e_{k-1}A_{m,k+1}A_{m,k-2}-e_{k+1}f_{k-1}A_{m,k}A_{m,k-1}}\right]\notag \\
   &+\left[{2e_{k}g_{k}A_{m,k+1}A_{m,k-1}-g_{k+1}e_{k-1}A_{m,k+2}A_{m,k-2}-e_{k+1}g_{k-1}A^{2}_{m,k}}\right]\notag \\
   &+\left[{2e_{k}h_{k}A_{m,k+2}A_{m,k-1}-h_{k+1}e_{k-1}A_{m,k+3}A_{m,k-2}-e_{k+1}h_{k-1}A_{m,k+1}A_{m,k}}\right]\notag \\
   &+\left[{2f_{k}g_{k}A_{m,k+1}A_{m,k}-g_{k+1}f_{k-1}A_{m,k+2}A_{m,k-1}-f_{k+1}g_{k-1}A_{m,k+1}A_{m,k}}\right]\notag \\
   &+\left[{2f_{k}h_{k}A_{m,k+2}A_{m,k}-h_{k+1}f_{k-1}A_{m,k+3}A_{m,k-1}-f_{k+1}h_{k-1}A_{m,k}A^{2}_{m,k+1}}\right]\notag \\
   &+\left[{2g_{k}h_{k}A_{m,k+2}A_{m,k+1}-h_{k+1}g_{k-1}A_{m,k+3}A_{m,k}-g_{k+1}h_{k-1}A_{m,k+2}A_{m,k+1}}\right].
\end{align}
In the following, we will examine the non-negativity of \eqref{eqq2}. On the one hand, from the log-concavity of $\left[A_{m,k}\right]_{0 \leq k \leq 2m}$, we deduce that
\begin{equation}\label{eqq3}
A_{m,k}^{2}-A_{m,k+1}A_{m,k-1} \geq 0.
\end{equation}
In addition, we know that the sequence $\left(\gamma_{k}\right)_{k\geq 0}$, is log-concave. Therefore, applying conditions (1)-(10) of Theorem \ref{thhm}, we obtain that
\begin{align}\label{eq4}
\eqref{eqq2}\geq & e_{k+1}e_{k-1}\left(A_{m,k-1}^{2}-A_{m,k}A_{m,k-2}\right)\notag \\
&+f_{k+1}f_{k-1}\left(A_{m,k}^{2}-A_{m,k+1}A_{m,k-1}\right) \notag \\
&+g_{k+1}g_{k-1}\left(A_{m,k+1}^{2}-A_{m,k+2}A_{m,k}\right)\notag \\
&+h_{k+1}h_{k-1}\left(A_{m,k+2}^{2}-A_{m,k+3}A_{m,k+1}\right)\notag \\ 
&+\gamma_{k+1}\gamma_{k-1}\left(A_{m,k-2}^{2}-A_{m,k-1}A_{m,k-3} \right)\notag \\
&+(2\gamma_{k}e_{k}-\gamma_{k+1}e_{k-1})A_{m,k-1}A_{m,k-2}-\gamma_{k-1}e_{k+1}A_{m,k}A_{m,k-3}\notag \\
&+\gamma_{k+1} f_{k-1}\left(A_{m,k}A_{m,k-2}-A^{2}_{m,k-1}\right)\notag \\
&+\gamma_{k-1} f_{k+1}\left(A_{m,k}A_{m,k-2}-A_{m,k+1}A_{m,k-3}\right)\notag \\
&+\gamma_{k+1}g_{k-1}\left(A_{m,k+1}A_{m,k-2}-A_{m,k}A_{m,k-1}\right)\notag \\
&+\gamma_{k-1}g_{k+1}\left(A_{m,k+1}A_{m,k-2}-A_{m,k+2}A_{m,k-3}\right)\notag \\
&+\gamma_{k+1}h_{k-1}\left(A_{m,k+2}A_{m,k-2}-A_{m,k+1}A_{m,k-1}\right)\notag \\
&+\gamma_{k-1}h_{k+1}\left(A_{m,k+2}A_{m,k-2}-A_{m,k+3}A_{m,k-3}\right)\notag \\
&+(2e_{k}f_{k}-e_{k+1}f_{k-1})A_{m,k}A_{m,k-1}-f_{k+1}e_{k-1}A_{m,k+1}A_{m,k-2}\notag\\
&+g_{k+1}e_{k-1}\left(A_{m,k+1}A_{m,k-1}-A_{m,k+2}A_{m,k-2}\right)\notag \\
&+e_{k+1}g_{k-1}\left(A_{m,k+1}A_{m,k-1}-A^{2}_{m,k}\right)\notag \\
&+h_{k+1}e_{k-1}\left(A_{m,k+2}A_{m,k-1}-A_{m,k+3}A_{m,k-2}\right)\notag \\
&+e_{k+1}h_{k-1}\left(A_{m,k+2}A_{m,k-1}-A_{m,k+1}A_{m,k}\right)\notag \\
&+(2f_{k}g_{k}-f_{k+1}g_{k-1})A_{m,k+1}A_{m,k}-g_{k+1}f_{k-1}A_{m,k+2}A_{m,k-1}\notag\\
&+h_{k+1}f_{k-1}\left(A_{m,k+2}A_{m,k}-A_{m,k+3}A_{m,k-1}\right)\notag\\
&+\left[{f_{k+1}h_{k-1}\left(A_{m,k+2}A_{m,k}-A^{2}_{m,k+1}\right)}\right]\notag\\
&+\left[{(2h_{k+1}g_{k-1}-g_{k+1}h_{k-1})A_{m,k+2}A_{m,k+1}-h_{k+1}g_{k-1}A_{m,k+3}A_{m,k}}\right].\notag
\end{align}
It follows that
\begin{align*}
  \eqref{eqq2}\geq & (e_{k+1}e_{k-1}-\gamma_{k+1} f_{k-1})\left(A_{m,k-1}^{2}-A_{m,k}A_{m,k-2}\right)\\
  &+\gamma_{k-1} f_{k+1}\left(A_{m,k}A_{m,k-2}-A_{m,k+1}A_{m,k-3}\right) 
           \end{align*}
   \begin{align*}
   &+(f_{k+1}e_{k-1}-\gamma_{k+1}g_{k-1})\left(A_{m,k}A_{m,k-1}- A_{m,k+1}A_{m,k-2}\right)\\
   &+\gamma_{k-1}g_{k+1}\left(A_{m,k+1}A_{m,k-2}-A_{m,k+2}A_{m,k-3}\right)\\
   &+(g_{k+1}e_{k-1}-\gamma_{k+1}h_{k-1})\left(A_{m,k+1}A_{m,k-1}-A_{m,k+2}A_{m,k-2}\right)\\
   &+\gamma_{k-1}h_{k+1}\left(A_{m,k+2}A_{m,k-2}-A_{m,k+3}A_{m,k-3}\right)\\
   &+(f_{k+1}f_{k-1}-e_{k+1}g_{k-1})\left(A_{m,k}^{2}-A_{m,k+1}A_{m,k-1}\right)\\
   &+\gamma_{k-1}e_{k+1}\left(A_{m,k-1}A_{m,k-2}-A_{m,k}A_{m,k-3}\right)\\
   &+(g_{k+1}f_{k-1}-e_{k+1}h_{k-1})\left(A_{m,k+1}A_{m,k}-A_{m,k+2}A_{m,k-1}\right)\\
   &+h_{k+1}e_{k-1}\left(A_{m,k+2}A_{m,k-1}-A_{m,k+3}A_{m,k-2}\right)\\
   &+(g_{k+1}g_{k-1}-f_{k+1}h_{k-1})\left(A_{m,k+1}^{2}-A_{m,k+2}A_{m,k}\right)\\
   &+h_{k+1}f_{k-1}\left(A_{m,k+2}A_{m,k}-A_{m,k+3}A_{m,k-1}\right)\\
   &+h_{k+1}h_{k-1}\left(A_{m,k+2}^{2}-A_{m,k+3}A_{m,k+1}\right)\\
   &+h_{k+1}g_{k-1}\left(A_{m,k+2}A_{m,k+1}-A_{m,k+3}A_{m,k}\right)\\
   &+\gamma_{k+1}\gamma_{k-1}\left(A_{m,k-2}^{2}-A_{m,k-1}A_{m,k-3} \right)\geq 0.
\end{align*}
Thus, we demonstrate that
\begin{equation*}
A_{m+1,k}^{2}-A_{m+1,k+1}A_{m+1,k-1}\geq 0
\end{equation*}
for any $ 0\leq k \leq m+1$, which completes the proof.
\end{proof}
The following special case of Theorem \ref{thhm} might be of particular interest.
\begin{corollary}
Define the matrix $\left[A_{n,k}\right]_{n,k\geq 0}$ recursively as follows:
\begin{align*}
  A_{n,0}= & \alpha A_{n-1,0}+gA_{n-1,1}+hA_{n-1,2}, \\
  A_{n,1}= & \beta A_{n-1,0}+fA_{n-1,1}+gA_{n-1,2}+hA_{n-1,3},\\
  A_{n,k}=&\gamma A_{n-1,k-2}+eA_{n-1,k-1}+fA_{n-1,k}+gA_{n-1,k+1}+hA_{n-1,k+2}, \text{\ \ for\ }n\geq 1, k\geq 2,
\end{align*}
where $A_{0,0}=1$ and $A_{0,k}=0$ for $k>0$. If the following conditions hold:
\begin{description}
  \item[(1)] $g^{2}\geq fh$ and $f\geq \alpha$,
  \item[(2)] $\beta^{2}\geq \alpha \gamma$ and $2\beta h \geq \alpha g$,
  \item[(3)] $fe\geq \gamma g$ and $fg \geq eh$,
  \item[(4)] $f^{2}\geq eg \geq \gamma h$ and  $e^{2}\geq \gamma f$,
  \item[(5)] $2\beta f \geq \alpha e +\gamma g$ and $2\beta g \geq ge +\gamma h$.
\end{description}
hen, for any fixed $n$, the row sequence $\left[A_{n,k}\right]_{0\leq k\leq 2n}$ is log-concave in $k$.
\end{corollary}
Applying the results of the previous corollary and theorem to some combinatorial arrays, we have the following proposition.
\begin{proposition}
Each row sequence in the $2$-Pascal triangle, the triangle \href{https://oeis.org/A291082}{A291082} and the triangle \href{https://oeis.org/A291080}{A291080} in OEIS \cite{slo} is log-concave.
\end{proposition}
By setting $\gamma=h=0$ and $e=1$ in the corollary above, we obtain the following result from Zhu \cite[Proposition 2.3]{XZ1}.
\begin{proposition}
Each row sequence in the Pascal triangle, Stirling triangle of the second kind, Aigner and Shapiro's Catalan triangles, Motzkin triangle, large Schr$\ddot{o}$der triangle, and Bell triangle is log-concave.
\end{proposition}
\section{Strong $q$-log-convexity of generating functions of rows} \label{Sec:3}
In this section, we establish the strong $q$-log-convexity of the generating functions for the rows of the generalized triangle defined in the previous section. To achieve this, we will rely on the following three lemmas.
\begin{lemma}\label{lm1}
A sufficient condition that a matrix is TP$_2$ is that all leading principal submatrices are TP$_2$.
\end{lemma}
\begin{lemma}\label{lm2}
If two matrices $\mathcal{N}$ and $\mathcal{M}$ are both TP$_2$, then we have that $\mathcal{N}\mathcal{M}$ is TP$_2$.
\end{lemma}
\begin{lemma}\label{lm}
If two matrices $\mathcal{N}$ and $\mathcal{M}$ are $q$-TP$_2$, then we have that $\mathcal{NM}$ is $q$-TP$_2$.
\end{lemma}

Now, we are prepared to present our theorem regarding the strong $q$-log-convexity of the generalized triangle.
\begin{theorem}\label{ttm}
Assume that the generalized triangular array $\left[A_{n,k}\right]_{n,k\geq 0}$ satisfies the recurrence relations:
\begin{align}
  A_{n,0}= & \alpha A_{n-1,0}+gA_{n-1,1}+hA_{n-1,2}, \notag \\
  A_{n,1}= & \beta A_{n-1,0}+fA_{n-1,1}+gA_{n-1,2}+hA_{n-1,3}, \notag\\
  A_{n,k}=&\gamma A_{n-1,k-2}+eA_{n-1,k-1}+fA_{n-1,k}+gA_{n-1,k+1}+hA_{n-1,k+2}, \text{\ for\ } n\geq 1,k\geq 2,\label{eq4}
\end{align}
with initial conditions $A_{0,0}=1$ and $A_{0,k}=0$ for $k>0$. If the following conditions hold:
\begin{description}
  \item[(1)] $f\geq \alpha$, $e\geq \beta$, $g\geq 0$ and $h\geq 0$,
  \item[(2)] $\alpha f\geq \beta g \geq \gamma h$ and $f^{2}\geq eg\geq \gamma h $,
  \item[(3)] $\alpha e\geq \gamma g$, $ef\geq \gamma g$ and $\beta f\geq \gamma g$,
  \item[(4)] $\beta e\geq \gamma f$, $\alpha g\geq \beta h$, $g^{2}\geq fh$ and $fg\geq eh$,
\end{description}
then the generating functions $A_{n}(q)=\sum_{k=0}^{2n}A_{n,k}q^{k}$ for $n\geq 0$ form a strongly $q$-log-convex sequence.
\end{theorem}
\begin{proof}
Define a triangular array $\mathcal{B}=[b_{n,k}(q)]_{n,k\geq 0}$ such that
$$b_{n,k}(q)=\sum_{i\geq k}A_{n,i}q^{i}\text{\ \ for all $n$ and $k$}.$$
Thus, by equation \eqref{eq4}, it is straightforward to verify that the triangular array $\mathcal{B}$ satisfies the recurrence given by
\begin{align}
  b_{n,k}(q)= & \gamma q^{2}b_{n-1,k-2}(q)+eqb_{n-1,k-1}(q)+fb_{n-1,k}(q) \notag \\
  &+\frac{g}{q}b_{n-1,k+1}(q)+\frac{h}{q^{2}}b_{n-1,k+2}(q)\text{\ \ \ for $n\geq 1$, $k\geq 2$,} \label{eq5}
  \end{align}
  \begin{align*}
  b_{n,0}(q)=  &(\alpha+\beta q+ \gamma q^{2})b_{n-1,0}(q)+\left[\frac{g}{q}+(f-\alpha)+(e-\beta)q\right]b_{n-1,1}(q) \notag\\
  &+\frac{h}{q^{2}}b_{n-1,2}(q),
\end{align*}

where $b_{0,0}=1$ and $b_{0,k}=0$ for $k>0$. From this recurrence relation, we can deduce that $A_{n}(q)=b_{n,0}(q)$.

\medskip

Thus, from Equation \ref{eq5}, it follows that
\begin{align*}
  &A_{n}(q)A_{m}(q)-A_{n+1}(q)A_{m-1}(q) \\
  &=b_{n,0}(q)b_{m,0}(q)-b_{n-1,0}(q)b_{m+1,0}(q)\\
  &=b_{n,0}(q)\left[ (\alpha+\beta q+ \gamma q^{2})b_{m-1,0}(q)+\left[\frac{g}{q}+(f-\alpha)+(e-\beta)q\right]b_{m-1,1}(q)\right.\\
  &\left.+\frac{h}{q^{2}}b_{m-1,2}(q)\right]\\
  &-b_{m-1,0}(q)\left[ (\alpha+\beta q+ \gamma q^{2})b_{n,0}(q)+\left[\frac{g}{q}+(f-\alpha)+(e-\beta)q\right]b_{n,1}(q)\right.\\
  &\left.+\frac{h}{q^{2}}b_{n,2}(q)\right]\\
  &=\left[g+(f-\alpha)q+(e-\beta)q^{2}\right]\left[b_{n,0}(q)b_{m-1,1}(q)-b_{m-1,0}(q)b_{n,1}(q)\right]/q\\
  &+h\left[b_{n,0}(q)b_{m-1,2}(q)-b_{m-1,0}(q)b_{n,2}(q)\right]/q^{2}
\end{align*}
for any $m\geq n$. To establish the strong $q$-log-convexity of the sequence of  $\left(A_{n}(q)\right)_{n\geq 0}$,  it suffices to demonstrate that the matrix $\mathcal{B}$ is $q$-TP$_{2}$, given that $f\geq \alpha$, $e\geq \beta$, $g\geq 0$, and $h\geq 0$.

\medskip

Consider the infinite lower triangular matrix
\begin{equation*}
\mathcal{A}=[A_{n,k}]_{k,n\geq 0}=\begin{bmatrix}
A_{0,0}\text{ \ \ } & 0 & 0 & 0 & 0 & \cdots\text{ \ } & 0   \\
A_{1,0}\text{ \ \ } & A_{1,1}\text{ \ } & A_{1,2}\text{ \ \ } & 0 &0 & \cdots\text{ \ } & 0  \\
A_{2,0}\text{ \ \ } & A_{2,1}\text{ \ } & A_{2,2}\text{ \ \ } & A_{2,3}\text{ \ } & A_{2,4}\text{ \ } & \cdots\text{ \ } & 0  \\
\vdots  & \vdots & \vdots & \vdots & \vdots & \ddots & \vdots
\end{bmatrix},
\end{equation*}
and an infinite lower triangular matrix $\mathcal{T}=[q^{i}]_{i\geq j\geq 0}$. It is clear that $\mathcal{B}=\mathcal{A}\mathcal{T}$ and that $\mathcal{T}$ is $q$-TP$_{2}$. This implies that $\mathcal{B}$ is also $q$-TP$_{2}$ by Lemma \ref{lm}, provided that $\mathcal{A}$ is $q$-TP$_{2}$. Therefore, we will now prove that $\mathcal{A}$ is $q$-TP$_{2}$.

\medskip

Define an infinite matrix
\begin{equation*}
\mathcal{J}=%
\begin{bmatrix}
\alpha\text{ \ } & \beta\text{ \ } & \gamma\text{ \ } & 0\text{\ } & 0\text{ \ } & \cdots\text{ \ } & 0 \\
g\text{ \ } & f\text{ \ } & e\text{ \ } & \gamma\text{\ } & 0\text{ \ } & \cdots\text{ \ } & 0 \\
h\text{ \ } & g\text{ \ } & f\text{ \ } & e\text{\ } & \gamma\text{ \ } & \cdots\text{ \ } & 0 \\
0\text{ \ } & h\text{ \ } & g\text{ \ } & f\text{ \ } & e\text{ \ } & \cdots\text{ \ } & 0  \\
0\text{ \ } & 0\text{ \ } & h\text{ \ } & g\text{ \ } & f\text{ \ } & \cdots\text{ \ } & 0  \\
\vdots\text{ \ }  & \vdots\text{ \ }  & \vdots\text{ \ }  & \vdots\text{ \ }  & \vdots  & \ddots\text{ \ } & \vdots
\end{bmatrix}.
\end{equation*}
Define the matrix $\overline{\mathcal{A}}$ as the matrix obtained from $\mathcal{A}$ by deleting its first row. Assume that the $n$th leading principal submatrices of $\mathcal{A}$, $\overline{\mathcal{A}}$, and $\mathcal{J}$ are denoted by $\mathcal{A}_{n}$, $\overline{\mathcal{A}}_{n}$, and $\mathcal{J}_{n}$, respectively.

\medskip

Thus, by \eqref{eq4} we have
\begin{equation*}
\overline{\mathcal{A}}_{n}=\mathcal{A}_{n}\mathcal{J}_{n}.
\end{equation*}
Note that $\mathcal{J}$ is TP$_2$ because the conditions (2)–(4) of Theorem \ref{ttm} are satisfied. Therefore, by induction on $n$, we conclude that 
$\mathcal{A}_n$ is TP$_2$ as per Lemma \ref{lm2}. Consequently, by applying Lemma \ref{lm1}, we find that $\mathcal{A}$ is also TP$_2$. This completes the proof.
\end{proof}
As applications of Theorem \ref{ttm} to triangular arrays, we obtain the following results.
\begin{proposition}
The generating functions of the rows in the triangles \href{https://oeis.org/A291082}{A291082} and \href{https://oeis.org/A291080}{A291080} form strongly $q$-log-convex sequences, respectively.
\end{proposition}
By setting $\gamma=h=0$ and $e=1$ in Theorem \ref{ttm}, we obtain the following result due to Zhu \cite[Proposition 3.5]{XZ1}.
\begin{proposition}
The generating functions of the rows in the Catalan triangles of Aigner and Shapiro, the Motzkin triangle, and the large Schr$\ddot{o}$der triangle form strongly $q$-log-convex sequences.
\end{proposition}
\section{Linear transformations preserving the strong $q$-log-convexity}\label{Sec:4}
Now, we will study some linear transformations that preserve the strong $q$-log-convexity and strong $q$-log-concavity properties.
\begin{lemma} \label{lm3}
Let $s$ be a positive integer. If the sequence $\left(x_{n}(q)\right)_{n}$ is strongly $q$-log-convex, then so is the sequence $\left(x_{k}(q)+x_{k+1}(q)+\cdots+ x_{k+s}(q)\right)_{k\geq 0}$.
\end{lemma}
\begin{proof}
Taking $s=1$, we have:
\begin{align*}
\left[x_{k-1}(q)+x_{k}(q)\right]\left[x_{l+1}(q)+x_{l+2}(q)\right]-\left[x_{k}(q)+x_{k+1}(q)\right]\left[x_{l}(q)+x_{l+1}(q)\right] &\\
=\left[x_{k-1}(q)x_{l+1}(q)-x_{k}(q)x_{l}(q)\right]+\left[x_{k}(q)x_{l+2}(q)-x_{k+1}(q)x_{l+1}(q)\right]&\\
+\left[x_{k-1}(q)x_{l+2}(q)-x_{k+1}(q)x_{l}(q)\right]\geq_{q} 0&
\end{align*}
for any $l\geq k\geq 0$, since $\left(x_{n}(q)\right)_{n}$ is strongly $q$-log-convex. Hence, by the induction hypothesis on $s$, we obtain the desired result.
\end{proof}

\begin{theorem}\label{tthm1}
The bi$^{s}$nomial transformation $B_{n}(q)=\sum_{k=0}^{sn}\binom{n}{k}_{s} f_{k}(q)$ preserves the strong $q$-log-convexity property.
\end{theorem}
\begin{proof}
Using Lemma \ref{lm3}, we will prove our result by induction on $n$ and $s$. If $0\leq n\leq 3$ and $s=2$, then we have
\begin{align*}
&B_{0}(q)=f_{0}(q),\\
&B_{1}(q)=f_{0}(q)+f_{1}(q)+f_{2}(q),\\
&B_{2}(q)=f_{0}(q)+2f_{1}(q)+3f_{2}(q)+2f_{3}(q)+f_{4}(q),\\
&B_{3}(q)=f_{0}(q)+3f_{1}(q)+6f_{2}(q)+7f_{3}(q)+6f_{4}(q)+3f_{5}(q)+f_{6}(q),
\end{align*}
since $B_{n}(q)=\sum_{k=0}^{sn}\binom{n}{k}_{s} f_{k}(q)$. Therefore, it follows from the strong $q$-log-convexity of $\left( f_{k}(q)\right)_{k\geq 0}$ that
\begin{align*}
&B_{0}(q)B_{2}(q)-B^{2}_{1}(q)=\left[f_{0}(q)f_{2}(q)-f^{2}_{1}(q)\right]+2\left[f_{0}(q)f_{3}(q)-f_{1}(q)f_{2}(q)\right]\\
&\  \ \ \ \ \ \ \ \ \ \ \ \ \ \ \ +\left[f_{0}(q)f_{4}(q)-f^{2}_{2}(q)\right]\geq_{q}0,
\end{align*}
\begin{align*}
&B_{0}(q)B_{3}(q)-B_{1}(q)B_{2}(q)=2\left[f_{0}(q)f_{2}(q)-f^{2}_{1}(q)\right]+5[f_{0}(q)f_{3}(q)\\
&-f_{1}(q)f_{2}(q)]+3\left[f_{0}(q)f_{4}(q)-f^{2}_{2}(q)\right]+2\left[f_{0}(q)f_{3}(q)-f_{1}(q)f_{3}(q)\right]\\
&+2\left[f_{0}(q)f_{5}(q)-f_{2}(q)f_{3}(q)\right]+\left[f_{0}(q)f_{5}(q)-f_{1}(q)f_{4}(q)\right]\\
&+\left[f_{0}(q)f_{6}(q)-f_{2}(q)f_{4}(q)\right]\geq_{q}0,
\end{align*}
\begin{align*}
&B_{1}(q)B_{3}(q)-B^{2}_{2}(q)=\left[f_{0}(q)f_{2}(q)-f^{2}_{1}(q)\right]+3\left[f_{0}(q)f_{3}(q)-f_{1}f(q)_{2}(q)\right]\\
&+3\left[f_{0}(q)f_{4}(q)-f^{2}_{2}(q)\right]+\left[f_{0}(q)f_{4}(q)-f_{1}(q)f_{3}(q)\right]\\
&+3\left[f_{0}(q)f_{5}(q)-f_{2}(q)f_{3}(q)\right]+\left[f_{0}(q)f_{6}(q)-f_{3}^{2}(q)\right]\\
&+2\left[f_{1}(q)f_{4}(q)-f_{2}(q)f_{3}(q)\right]+3\left[f_{1}(q)f_{5}(q)-f_{3}^{2}(q)\right]\\
&+\left[f_{1}(q)f_{6}(q)-f_{3}(q)f_{4}(q)\right]+3\left[f_{1}(q)f_{6}(q)-f_{2}(q)f_{5}(q)\right]\geq_{q}0,
\end{align*}
which implies that $B_{0}(q),B_{1}(q),B_{2}(q),B_{3}(q)$ is strongly $q$-log-convex. So we proceed to the inductive step ($n\geq 4$ and $s\geq 3$).

\medskip

Note that
\begin{align*}
&B_{n}(q)=\sum_{k=0}^{sn}\binom{n}{k}_{s} f_{k}(q)\\
&B_{n}(q)=\sum_{k=0}^{s(n-1)}\binom{n-1}{k}_{s}\sum_{j=0}^{s} f_{k+j}(q).
\end{align*}
Therefore, by the induction hypothesis and the strong $q$-log-convexity of the sequence
$\left(x_{k}(q)+x_{k+1}(q)+\cdots+ x_{k+s}(q)\right)_{k\geq 0}$, we conclude that the sequence $B_{0}(q),B_{1}(q),\ldots, B_{n}(q)$ is strongly $q$-log-convex.
\end{proof}

Like the proof of Theorem \ref{tthm1}, we also have the following result.
\begin{theorem}\label{tthm2}
The bi$^{s}$nomial transformation $B_{n}(q)=\sum_{k=0}^{sn}\binom{n}{k}_{s} f_{k}(q)$ preserves the strong $q$-log-concavity property.
\end{theorem}

\section*{Acknowledgement(s)}
The authors would like to thank the referees for many valuable remarks andsuggestions to improve the original manuscript. This work was supported by DG-RSDT (Algeria), PRFU Project, No. C00L03UN180120220002.

\end{document}